\documentclass[12pt]{article}%
\usepackage{amsmath}
\usepackage{amsfonts}
\usepackage{amssymb}
\usepackage{graphicx}%
\providecommand{\U}[1]{\protect\rule{.1in}{.1in}}
\newtheorem{theorem}{Theorem}

\newtheorem{algorithm}[theorem]{Algorithm}

\newtheorem{condition}[theorem]{Condition}

\newtheorem{definition}[theorem]{Definition}

\newtheorem{lemma}[theorem]{Lemma}

\newenvironment{proof}[1][Proof]{\noindent\textbf{#1.} }{\ \rule{0.5em}{0.5em}}
\begin{document}

\title{\textbf{String-Averaging Projected Subgradient Methods for Constrained
Minimization}}
\author{Yair Censor$^{1}$ and Alexander J. Zaslavski$^{2}\bigskip$\\$^{1}$Department of Mathematics, University of Haifa\\Mt.\ Carmel, Haifa 31905, Israel\\(yair@math.haifa.ac.il) \bigskip\\$^{2}$Department of Mathematics\\The Technion -- Israel Institute of Technology\\Technion City, Haifa 32000, Israel\\(ajzasl@techunix.technion.ac.il)}
\date{February 17, 2013. \\
Revised: June 28, 2013, and August 26, 2013.}
\maketitle

\begin{abstract}
We consider constrained minimization problems and propose to replace the
projection onto the entire feasible region, required in the Projected
Subgradient Method (PSM), by projections onto the individual sets whose
intersection forms the entire feasible region. Specifically, we propose to
perform such projections onto the individual sets in an algorithmic regime of
a feasibility-seeking iterative projection method. For this purpose we use the
recently developed family of Dynamic String-Averaging Projection (DSAP)
methods wherein iteration-index-dependent variable strings and variable
weights are permitted. This gives rise to an algorithmic scheme that
generalizes, from the algorithmic structural point of view, earlier work of
Helou Neto and De Pierro, of Nedi\'{c}, of Nurminski, and of Ram et al.

\end{abstract}

\textbf{Keywords and phrases}: Fixed point, Hilbert space, metric projection,
nonexpansive operator, perturbation resilience, projected subgradient
minimization, string-averaging projection methods, superiorization method,
variable strings, variable weights.\bigskip

\textbf{2010 Mathematics Subject Classification (MSC)}: 65K10, 90C25

\section{Introduction\label{sect:intro}}

\textbf{The problem}: We consider constrained minimization problems of the
form%
\begin{equation}
\operatorname*{minimize}\{\phi(x)\mid x\in C\},\label{prob:cons-min}%
\end{equation}
where $\phi$ is a convex objective function mapping from the $J$-dimensional
Euclidean space $R^{J}$ into the reals and $C\subseteq R^{J}$ is a given
closed convex constraints set. Such constrained optimization problems lie at
the heart of optimization theory and practice and constitute mathematical
models for many scientific and real-world applications. The Projected
Subgradient Method (PSM) for constrained minimization interlaces subgradient
steps for objective function descent with projections onto the feasible region
$C$ to regain feasibility after every subgradient step.

PSM generates a sequence of iterates $\{x^{k}\}_{k=0}^{\infty}$ according to
the recursion formula%
\begin{equation}
x^{k+1}=P_{C}\left(  x^{k}-t_{k}\phi^{\prime}(x^{k})\right)  , \label{eq:sgp}%
\end{equation}
where $t_{k}>0$ is a step-size, $\phi^{\prime}(x^{k})\in\partial\phi(x^{k})$
is a subgradient of $\phi$ at $x^{k},$ and $P_{C}$ stands for the orthogonal
(least Euclidean norm) projection onto the set $C.$ The underlying philosophy
is to perform unconstrained objective function descent steps by moving from
$x^{k}$ to $z^{k}:=x^{k}-t_{k}\phi^{\prime}(x^{k})$ and then regain
feasibility with respect to $C$ by projecting $z^{k}$ onto $C.$

\textbf{Motivation}: This projection onto $C$ is a computational bottleneck in
applying the method due to the required inner loop of (quadratic) minimization
of the distance to the set $C.$ Therefore the projected subgradient method is
mostly useful only when the feasible region is \textquotedblleft simple to
project onto\textquotedblright. To alleviate this difficulty we consider the
situation where $C_{1},C_{2},\dots,C_{m}$ are nonempty closed convex subsets
of $R^{J},$ where $m$ is a natural number, define%
\begin{equation}
C:=\cap_{i=1}^{m}C_{i},\label{eq:1.1}%
\end{equation}
and propose to replace the projection $P_{C}$ in (\ref{eq:sgp}) by projections
onto the individual sets $C_{i}$ executed in a specific algorithmic regime of
a \textit{feasibility-seeking} iterative projection method. This is an
important development because often the entire feasible region $C$ is not
\textquotedblleft simple to project onto\textquotedblright\ whereas the
individual sets are. The class of \textit{projection methods} is understood
here as the class of methods that have the feature that they can reach an aim
related to the family of sets $\{C_{1},C_{2},\dots,C_{m}\}$ by performing
projections (orthogonal, i.e., least Euclidean distance, or others) onto the
individual sets $C_{i}.$ The advantage of such methods occurs in situations
where projections onto the individual sets are computationally simple to
perform or at least simpler than a projection on the intersection $C$. Such
methods have been in recent decades extensively investigated mathematically
and used experimentally with great success on some huge and sparse real-world
applications, consult, e.g., \cite{bb96,cccdh10}, the books
\cite{BC11,byrnebook,CEG12,CZ97,chinneck-book,ER11,galantai,GTH}, and the
recent paper \cite{bk13}.

\textbf{Contribution}: We employ here for the feasibility-seeking method the
\textit{String-Averaging Projection} (SAP) scheme. This class was first
introduced in \cite{ceh01} (in a formulation that is not restricted to
feasibility-seeking) and was subsequently studied further in
\cite{bmr04,bdhk07,CS08,CS09,ct03,crombez}, see also \cite[Example 5.20]%
{BC11}. SAP methods were also employed in applications \cite{pen09,rhee03}.
The Component-Averaged Row Projections (CARP) method of \cite{gordon} also
belongs to the class of SAP methods. Within the class of projection methods,
SAP methods do not constitute a single algorithm but rather an
\textit{algorithmic scheme, }which means that by making a specific choice of
strings and weights in SAP, along with choices of other parameters in the
scheme, a deterministic algorithm for the problem at hand can be obtained.

Here we use the recently developed family of \textit{Dynamic String-Averaging
Projection} (DSAP) methods of \cite{cz12} wherein iteration-index-dependent
variable strings and variable weights are permitted (in all works prior to
\cite{cz12}, SAP methods were formulated with a single predetermined set of
strings and a single predetermined set of weights.)

\textbf{Relation with previous works: }The literature on subgradient
minimization is vast and we mention \cite{kiwiel, polyak} as reperesentatives.
Specifically to the research presented here, the introduction of DSAP to
replace the projection $P_{C}$ in (\ref{eq:sgp}) gives rise to an algorithmic
scheme that includes and generalizes, from the algorithmic structural point of
view, earlier recent work of Helou Neto and De Pierro
\cite{hdp-siam09,hdp11-1}, of Nedi\'{c} \cite{nedic-mp11}, of Nurminski
\cite{nur08b,nur08a,nur10,nur11}, and of Ram et al. \cite{nedic-et-al}. 

\textbf{Nurminski:} The algorithms of Nurminski use Fej\'{e}r operators, that
can be used in feasibility-seeking, and introduces into them disturbances with
diminishing step-sizes $\lambda_{k}\rightarrow0$ as $k\rightarrow\infty$,
where the rate of this tendency is such that $\sum_{k=0}^{\infty}\lambda
_{k}=+\infty.$ Under these conditions and a variety of additional assumptions,
Nurminski showed asymptotic convergence of the iterates generated by his
algorithms to a minimum point of the constrained minimization problem.

\textbf{Helou Neto and De Pierro:} The framework proposed by Helou Neto and De
Pierro uses interlacing of {}\textquotedblleft feasibility
operators\textquotedblright{}\ with {}\textquotedblleft optimality
operators\textquotedblright{}\ with the aim of creating exact constrained
minimization algorithms. Similarly to Nurminski, they employ diminishing
step-sizes$.$ Under these conditions and a variety of additional assumptions,
different than those of Nurminski, they show asymptotic convergence of the
iterates generated by their algorithmic framework to a minimum point of the
constrained minimization problem.

However, when it comes to derivation of specific algorithms from the general
framework of \cite[Equation (3)]{hdp11-1}, their feasibility operator $F$
invariably takes the form
\begin{equation}
\mathcal{F}_{F}(x)=x-\mu(x)\nabla F(x),\label{eq:hdp-feas-op}%
\end{equation}
where the function $F(x),$ whose gradient is calculated, is \textquotedblleft
a convex function such that the set of minima of this function coincides with
the set $X$ [in \cite[Equation (3)]{hdp11-1} $X$ is the feasible set of the
minimization problem and should be identified with $C=\cap_{i=1}^{m}C_{i}$ in
our notation] when it is not empty and defines a solution in an appropriate
way (least squares for example) otherwise\textquotedblright{}\ and $\mu(x)$
are some parameters that are restricted in a particular manner as in
\cite[Lemma 7 or Corollary 8]{hdp11-1}. Our feasibility-seeking DSAP
algorithms are not limited to the form of (\ref{eq:hdp-feas-op}).

\textbf{Nedi\'{c} and Ram et al.:} The overall approach here is to apply
gradient and subgradient iterative methods for the objective function
minimization and interlace into them random feasibility updates. The resulting
\textquotedblleft random projection method\textquotedblright\ \cite[Equation
(4)]{nedic-mp11} bears structural similarity to our approach since it replaces
the projection $P_{C}$ in (\ref{eq:sgp}) by projections onto individual sets
$C_{i}$ but not in a DSAP regime which is more general. The randomness refers
to the way the constraints are picked up for the feasibility updates.

There are various differences among all the above works and between them and
our work, differences in overall setup of the problems, differences in the
assumptions used for the various convergence results, etc. This is not the
place for a full review of all these differences but the main contribution of
our work presented here lies in the many more algorithmic possibilities
revealed by the use of DSAP general scheme, while retaining overall
convergence to an (exact) solution of the problem (\ref{prob:cons-min}).

\textbf{Paper layout}: The DSAP algorithmic scheme is presented in Section
\ref{sect:SAPv} and our \textit{String-Averaging Projected Subgradient Method
}(SA-PSM) is presented in Section \ref{sec:sa-psm} and its convergence
analysis is done in Section \ref{sec:convergence}.

\section{The dynamic string-averaging projection \newline method with variable
strings and weights\label{sect:SAPv}}

Our main result in Theorem \ref{thm:7.1} below is obtained in the
finite-dimensional Euclidean space but the algorithms described in this
section can be used in the more general setting of an infinite-dimensional
Hilbert space. We first present, for the reader's convenience and for the sake
of completeness, the string-averaging algorithmic scheme of \cite{ceh01}. Let
$H$ be a real Hilbert space with inner product $\left\langle \cdot
,\cdot\right\rangle $ and induced norm $||\cdot||$. For $t=1,2,\dots,M,$ let
the \textit{string} $I_{t}$ be an ordered subset of $\{1,2,\dots,m\}$ of the
form%
\begin{equation}
I_{t}=(i_{1}^{t},i_{2}^{t},\dots,i_{m(t)}^{t}),\label{block}%
\end{equation}
with $m(t)$ the number of elements in $I_{t}.$ Suppose that there is a set
$S\subseteq H$ such that there are operators $O_{1},O_{2},\dots,O_{m}$ mapping
$S$ into $S$ and an additional operator $O$ which maps $S^{M}$ into $S$.

\begin{algorithm}
\label{alg:sa}\textbf{The string-averaging algorithmic scheme}

\textbf{Initialization}: $x^{0}\in S$ is arbitrary.

\textbf{Iterative Step}: given the current iterate $x^{k}$,

(i) calculate, for all $t=1,2,\dots,M,$%
\begin{equation}
T_{t}(x^{k})=O_{i_{m(t)}^{t}}\cdots O_{i_{2}^{t}}O_{i_{1}^{t}}(x^{k}),
\label{as1}%
\end{equation}
(ii) and then calculate%
\begin{equation}
x^{k+1}=O(T_{1}(x^{k}),T_{2}(x^{k}),\ldots,T_{M}(x^{k})). \label{as2}%
\end{equation}

\end{algorithm}

For every $t=1,2,\dots,M,$ this algorithmic scheme applies to $x^{k}$
successively the operators whose indices belong to the $t$th string. This can
be done in parallel for all strings and then the operator $O$ maps all
end-points $T_{t}(x^{k})$ onto the next iterate $x^{k+1}$. This is indeed an
algorithm provided that the operators $\{O_{i}\}_{i=1}^{m}$ and $O$ all have
algorithmic implementations. In this framework we get a (fully)
\textit{sequential algorithm} by the choice $M=1$ and $I_{1}=(1,2,\dots,m)$ or
a (fully) \textit{simultaneous algorithm} by the choice $M=m$ and
$I_{t}=(t),\ t=1,2,\dots,M.$ Originally \cite{ceh01}, Algorithm \ref{alg:sa}
can employ operators other than projections and convex combinations, therefore
it is more general than SAP. On the other hand, Algorithm \ref{alg:sa} is
formulated for fixed strings and weights and this is generalized in the
formulation of DSAP that follows. 

Next we describe the \textit{dynamic string-averaging projection} (DSAP)
method with variable strings and weights. For $i=1,2,\dots,m,$ we denote the
projection onto the set $C_{i}$ by $P_{i}=P_{C_{i}}.$ An \textit{index vector}
is a vector $t=(t_{1},t_{2},\dots,t_{q})$ such that $t_{i}\in\{1,2,\dots,m\}$
for all $i=1,2,\dots,q$, its \textit{length} is denoted by $\ell(t)=q,$ and we
define the operator $P[t]$ as the product of the individual projections onto
the sets whose indices appear in the index vector $t$, namely,%
\begin{equation}
\;P[t]:=P_{t_{q}}P_{t_{q-1}}\cdots P_{t_{1}},
\end{equation}
and call it a \textit{string operator}.

\begin{definition}
\label{def:fit}A finite set $\Omega$ of index vectors is called \texttt{fit
}if for each $i\in\{1,2,\dots,m\}$, there exists a vector $t=(t_{1}%
,t_{2},\dots,t_{q})\in\Omega$ where $q$ is a natural number such that
$t_{s}=i$ for some $s\in\{1,2,\dots,q\}$.
\end{definition}

Note that in the above definition $q$ can vary. For each index vector $t$ the
string operator is nonexpansive, since the individual projections are, i.e.,%
\begin{equation}
\left\Vert P[t](x)-P[t](y)\right\Vert \leq||x-y||\text{ for all }x,y\in
H,\label{eq:1.8}%
\end{equation}
and also%
\begin{equation}
P[t](x)=x\text{ for all }x\in C.\label{eq:1.9}%
\end{equation}

Denote by $\mathcal{M}$ the collection of all pairs $(\Omega,w)$, where
$\Omega$ is a fit finite set of index vectors and%
\begin{equation}
w:\Omega\rightarrow(0,\infty)\text{ is such that }\sum_{t\in\Omega}w(t)=1.
\label{eq:1.6}%
\end{equation}

A pair $(\Omega,w)\in\mathcal{M}$ and the function $w$ were called in
\cite{bdhk07} an \textit{amalgamator} and a \textit{fit weight function},
respectively. For any $(\Omega,w)\in\mathcal{M}$ define the convex combination
of the end-points of all strings defined by members of $\Omega$ by%
\begin{equation}
P_{\Omega,w}(x):=\sum_{t\in\Omega}w(t)P[t](x),\;x\in H. \label{eq:1.7}%
\end{equation}

Fix a number $\Delta\in(0,1/m)$ and an integer $\bar{q}\geq m$ and denote by
$\mathcal{M}_{\ast}\equiv\mathcal{M}_{\ast}(\Delta,\bar{q})$ the set of all
$(\Omega,w)\in\mathcal{M}$ such that the lengths of the strings are bounded
and the weights are bounded away from zero, namely,%
\begin{equation}
\mathcal{M}_{\ast}:=\{(\Omega,w)\in\mathcal{M\mid}\text{ }\ell(t)\leq\bar
{q}\text{ and }\Delta\leq w(t),\text{ for all }t\in\Omega\}.\label{eq:11516}%
\end{equation}
We make the assumption that $\bar{q}\geq m$ because only in this case there
exists $(\Omega,w)\in M_{\ast}$ such that $\Omega$ is a singleton containing
the index vector $(1,2,\dots,m)$ so that our class of algorithms includes also
the classical cyclic projection algorithm.

The dynamic string-averaging projection (DSAP) method with variable strings
and variable weights can now be described by the following algorithm.

\begin{algorithm}
\label{alg:sap-v}$\left.  {}\right.  $\textbf{The DSAP method with variable
strings and variable weights}

\textbf{Initialization}: select an arbitrary $x^{0}\in H$,

\textbf{Iterative step}: given a current iteration vector $x^{k}$ pick a pair
$(\Omega_{k},w_{k})\in\mathcal{M}_{\ast}$ and calculate the next iteration
vector $x^{k+1}$ by%
\begin{equation}
x^{k+1}=P_{\Omega_{k},w_{k}}(x^{k})\text{.\label{eq:algv}}%
\end{equation}

\end{algorithm}

The convergence properties and the, so called, perturbation resilience of this
DSAP method were analyzed in \cite{cz12}.

\section{The string-averaging projected subgradient method\label{sec:sa-psm}}

Our proposed string-averaging projected subgradient method (SA-PSM) for the
solution of (\ref{prob:cons-min}) performs string-averaging steps with respect
to the individual constraints of (\ref{eq:1.1}) instead of the single
projection onto the entire feasible set $C$ of (\ref{prob:cons-min}) dictated
by the PSM of (\ref{eq:sgp}).

\begin{algorithm}
\label{alg:sa-psm}$\left.  {}\right.  $\textbf{The string-averaging projected
subgradient method (SA-PSM)}

\textbf{(0) Initialization}: Let $\{\alpha_{k}\}_{k=0}^{\infty}\subset(0,1]$
be a scalar sequence and select arbitrary vectors $x^{0},$ $s^{0}\in H$,

\textbf{(1)} \textbf{Iterative step}: given a current iteration vector $x^{k}$
and a current vector $s^{k},$ pick a pair $(\Omega_{k},w_{k})\in
\mathcal{M}_{\ast}$ and calculate the next vectors as follows:

\textbf{(1.1) }if $0\in\partial\phi(x^{k})$ then set $s^{k}=0$ and calculate%
\begin{equation}
x^{k+1}=P_{\Omega_{k},w_{k}}(x^{k})\text{.\label{eq:alg-sa-psm-1}}%
\end{equation}

\textbf{(1.2)} if $0\notin\partial\phi(x^{k})$ then choose and set $s^{k}%
\in\partial\phi(x^{k})$ and calculate%
\begin{equation}
x^{k+1}=P_{\Omega_{k},w_{k}}\left(  x^{k}-\alpha_{k}\frac{s^{k}}{\parallel
s^{k}\parallel}\right)  \text{.\label{eq:alg-sa-psm-2}}%
\end{equation}

\end{algorithm}

For each $x\in H$ and nonempty set $E\subseteq H$ define the distance%
\begin{equation}
d(x,E)=\inf\{||x-y||\mid\;y\in E\}.
\end{equation}
Denoting the solution set of (\ref{prob:cons-min}) by%
\begin{equation}
SOL(\phi,C):=\{x\in C\mid\;\phi(x)\leq\phi(y){\text{ for all }}y\in C\},
\end{equation}
we specify, in our convergence result presented in the sequel, conditions
which guarantee that for every $\varepsilon\in(0,1),$ and any sequence
$\{x^{k}\}_{k=0}^{\infty}$, generated by Algorithm \ref{alg:sa-psm}, there
exists an integer $K$ so that, for all $k\geq K,$ the following inequalities
hold:%
\begin{gather}
d(x^{k},SOL(\phi,C))\leq\varepsilon{\text{ }}\\
{\text{and }}\nonumber\\
\phi(x^{k})\leq\operatorname{inf}\{\phi(z)\mid z\in SOL(\phi,C)\}+\varepsilon
\bar{L},
\end{gather}
where $\bar{L}$ is some well-defined constant. For each $x\in H$ and $r>0$
define the closed ball%
\begin{equation}
B(x,r)=\{y\in H\mid\;||x-y||\leq r\}.
\end{equation}

\begin{definition}
\label{def:M-fit}Assume that $C_{1},C_{2},\dots,C_{m}$ are nonempty closed
convex subsets of $R^{J},$ $C=\cap_{i=1}^{m}C_{i},$ and that $C$ is nonempty.
We call a finite set $\Omega$ of index vectors $M$\texttt{-fit with respect to
the family }$\{C_{1},C_{2},\dots,C_{m}\}$ ($M$\texttt{-fit},\texttt{ }for
short)\texttt{ }if $\Omega$ is fit (see Definition \ref{def:fit}) and there
exists an $M>0$ such that for each $t=(t_{1},t_{2},\dots,t_{\ell(t)})\in
\Omega$ there is a $u\in\{1,2,\dots,\ell(t)\}$ such that%
\begin{equation}
C_{t_{u}}\subset B(0,M).
\end{equation}
Denote the set $\mathcal{M}_{\ast\ast}:=\{(\Omega,w)\in\mathcal{M}_{\ast}%
\mid\Omega$ is $M$-fit$\}$.
\end{definition}

In the next theorem we establish the convergence of sequences generated by
Algorithm \ref{alg:sap-v} with computational errors.

\begin{theorem}
\label{thm:thm5.1}Assume that $C_{1},C_{2},\dots,C_{m}$ are nonempty closed
convex subsets of $R^{J},$ $C=\cap_{i=1}^{m}C_{i}$ and $C$ is nonempty. For
some $M>0$ let there exist an index $s\in\{1,2,\dots,m\}$ such that%
\begin{equation}
C_{s}\subset B(0,M). \label{eq:boundedness}%
\end{equation}
Further, let%
\begin{equation}
\{\gamma_{k}\}_{k=1}^{\infty}\subset(0,1]\text{ so that}\;\lim_{k\rightarrow
\infty}\gamma_{k}=0, \label{eq:5.3}%
\end{equation}
and $\varepsilon>0$. Then there exists an integer $K$ such that for any
$\{(\Omega_{k},w_{k})\}_{k=1}^{\infty}\subset\mathcal{M}_{\ast\ast}$ and any
sequence $\{x^{k}\}_{k=0}^{\infty}$, which satisfies for all integer $k\geq1$,%
\begin{equation}
\left\Vert x^{k}-P_{\Omega_{k},w_{k}}(x^{k-1})\right\Vert \leq\gamma_{k},
\end{equation}
the inequality $d(x^{k},C)\leq\varepsilon$ holds for all integers $k\geq K$.
\end{theorem}

In order to prove Theorem \ref{thm:thm5.1} we need the following auxiliary lemma.

\begin{lemma}
\label{lem:6.1}{Under the assumptions of Theorem \ref{thm:thm5.1},} if
$\{(\Omega_{k},w_{k})\}_{k=1}^{\infty}\subset\mathcal{M}_{\ast\ast}$ and
$\{x^{k}\}_{k=0}^{\infty}$ satisfies for all integer $k\geq1$,
\begin{equation}
\left\Vert x^{k}-P_{\Omega_{k},w_{k}}(x^{k-1})\right\Vert \leq1,
\label{eq:6.2}%
\end{equation}
then $\left\Vert x^{k}\right\Vert \leq3M+1$ for all $k\geq1$.
\end{lemma}

\begin{proof}
Let $z\in C,$ then, by (\ref{eq:boundedness}),
\begin{equation}
\left\Vert z\right\Vert \leq M. \label{eq:zm}%
\end{equation}
For any $k\geq0$ (\ref{eq:1.7}) implies
\begin{equation}
P_{\Omega_{k+1},w_{k+1}}(x^{k})=\sum_{t\in\Omega_{k+1}}w_{k+1}(t)P[t](x^{k}).
\label{eq:6.5}%
\end{equation}
By (\ref{eq:6.2}), (\ref{eq:zm}), (\ref{eq:6.5}) and (\ref{eq:1.6}),%
\begin{align}
\left\Vert x^{k+1}\right\Vert  &  \leq\left\Vert z\right\Vert +\left\Vert
x^{k+1}-z\right\Vert \leq M+\left\Vert x^{k+1}-z\right\Vert \nonumber\\
&  \leq M+\left\Vert z-P_{\Omega_{k+1},w_{k+1}}(x^{k})\right\Vert +\left\Vert
P_{\Omega_{k+1},w_{k+1}}(x^{k})-x^{k+1}\right\Vert \nonumber\\
&  \leq M+1+\parallel z-\sum_{t\in\Omega_{k+1}}w_{k+1}(t)P[t](x^{k}%
)\parallel\nonumber\\
&  \leq M+1+\sum_{t\in\Omega_{k+1}}w_{k+1}(t)\left\Vert z-P[t](x^{k}%
)\right\Vert . \label{eq:6.6}%
\end{align}
Let $t=(t_{1},t_{2},\dots,t_{\ell(t)})\in\Omega_{k+1},$ then%
\begin{equation}
P[t](x^{k})=P_{t_{\ell(t)}}P_{t_{\ell(t)-1}}\cdots P_{t_{1}}(x^{k}),
\label{eq:6.8}%
\end{equation}
and, since $\{(\Omega_{k},w_{k})\}_{k=1}^{\infty}\subset\mathcal{M}_{\ast\ast
},$ there is an index $j\in\{1,2,\dots,\ell(t)\}$ for which%
\begin{equation}
C_{t_{j}}\subset B(0,M). \label{eq:6.9}%
\end{equation}
From (\ref{eq:6.9}),%
\begin{equation}
\left\Vert P_{t_{j}}P_{t_{j-1}}\cdots P_{t_{1}}(x^{k})\right\Vert \leq M.
\label{eq:6.10}%
\end{equation}
Thus, since $z\in C$ and by (\ref{eq:zm}), (\ref{eq:6.8}), (\ref{eq:6.10}) and
properties of the projection operator we have%
\begin{equation}
\left\Vert z-P[t](x^{k})\right\Vert \leq\left\Vert z-P_{t_{j}}P_{t_{j-1}%
}\cdots P_{t_{1}}(x^{k})\right\Vert \leq2M, \label{eq:6.11}%
\end{equation}
yielding, by (\ref{eq:6.6}) and (\ref{eq:6.11}),
\begin{equation}
\left\Vert x^{k+1}\right\Vert \leq3M+1,
\end{equation}
which completes the proof of the lemma.
\end{proof}

Now we are ready to prove {Theorem \ref{thm:thm5.1}.}

\textbf{Proof of Theorem \ref{thm:thm5.1}}{. }We use \cite[Theorem 7]{cz12}.
All its assumptions hold here. Indeed, Assumption (i) of \cite[Theorem
7]{cz12} follows from (\ref{eq:boundedness}), Assumption (ii) holds since our
space is finite-dimensional, and Assumptions (iii) and (iv) obviously hold.
Therefore, there exists an integer $k_{1}$ so that, for each $\{(\Omega
_{k},w_{k})\}_{k=1}^{\infty}\subset\mathcal{M}_{\ast\ast},$ any sequence
$\{x^{k}\}_{k=0}^{\infty}$, generated by Algorithm \ref{alg:sap-v} with
$\left\Vert x^{0}\right\Vert \leq3M+1,$ converges, $\lim_{s\rightarrow\infty
}x^{s}\in C$, and%
\begin{equation}
\Vert x^{k}-\lim_{s\rightarrow\infty}x^{s}\Vert\leq\varepsilon/4{\text{ for
all }}k\geq k_{1}.
\end{equation}
By (\ref{eq:5.3}) there is an integer $k_{2}$ such that%
\begin{equation}
\gamma_{k}\leq(\varepsilon/4)(k_{1}+1)^{-1}{\text{ for all }}k\geq
k_{2}.\label{eq:6.12}%
\end{equation}
Define%
\begin{equation}
K:=k_{1}+k_{2}+2.\label{eq:6.13}%
\end{equation}
We will show that $d(x^{k},C)\leq\varepsilon$ for all $k\geq K.$ To this end
take some $v\geq K$ and let%
\begin{equation}
y^{0}:=x^{v-k_{1}}.\label{eq:6.16}%
\end{equation}
For any $k\geq1$ let%
\begin{equation}
y^{k}:=x^{k+v-k_{1}}\text{ and}\;(\tilde{\Omega}_{k},\tilde{w}_{k}%
):=(\Omega_{k+v-k_{1}},w_{k+v-k_{1}}).\label{eq:6.17}%
\end{equation}
By (\ref{eq:5.3}), and since $\{(\Omega_{k},w_{k})\}_{k=1}^{\infty}%
\subset\mathcal{M}_{\ast\ast}$ and the sequence $\{x^{k}\}_{k=0}^{\infty}$
satisfies%
\begin{equation}
\left\Vert x^{k}-P_{\Omega_{k},w_{k}}(x^{k-1})\right\Vert \leq\gamma
_{k},\text{ for all }k\geq1,
\end{equation}
we can use Lemma \ref{lem:6.1} to obtain%
\begin{equation}
\left\Vert x^{k}\right\Vert \leq3M+1{\text{ for all }}k\geq1.\label{eq:6.18}%
\end{equation}
Therefore,
\begin{equation}
\left\Vert y^{0}\right\Vert =\left\Vert x^{v-k_{1}}\right\Vert \leq
3M+1.\label{eq:6.19}%
\end{equation}
It follows that, for each $k\geq1$,%
\begin{align}
\left\Vert y^{k}-P_{\tilde{\Omega}_{k},\tilde{w}_{k}}(y^{k-1})\right\Vert  &
=\left\Vert x^{k+v-k_{1}}-P_{\Omega_{k+v-k_{1}},w_{k+v-k_{1}}}(x^{k+v-k_{1}%
-1})\right\Vert \nonumber\\
&  \leq\gamma_{k+v-k_{1}}\leq(\varepsilon/4)(k_{1}+1)^{-1}.\label{eq:6.20}%
\end{align}
Defining now%
\begin{equation}
z^{0}:=y^{0}\text{ and }z^{k}=P_{\tilde{\Omega}_{k},\tilde{w}_{k}}%
(z^{k-1})\text{ for all }k\geq1\label{eq:6.21.22}%
\end{equation}
we obtain that%
\begin{equation}
d(z^{k_{1}},C)\leq\varepsilon/4.\label{eq:6.23}%
\end{equation}
Next we show, by induction, that for all $k=0,1,\dots,k_{1}$%
\begin{equation}
\left\Vert y^{k}-z^{k}\right\Vert \leq k(\varepsilon/4)(k_{1}+1)^{-1}%
.\label{eq:6.24}%
\end{equation}
Clearly, for $k=0$ (\ref{eq:6.24}) holds. Assume that $0\leq k<k_{1}$ and that
(\ref{eq:6.24}) holds. Then, by (\ref{eq:6.20}), (\ref{eq:6.21.22}),
(\ref{eq:6.24}) and nonexpansivity of $P_{\Omega,w}$ of (\ref{eq:1.7}),%
\begin{align}
\left\Vert y^{k+1}-z^{k+1}\right\Vert  &  =\left\Vert y^{k+1}-P_{\tilde
{\Omega}_{k+1},\tilde{w}_{k+1}}(z^{k})\right\Vert \nonumber\\
&  \leq\left\Vert y^{k+1}-P_{\tilde{\Omega}_{k+1},\tilde{w}_{k+1}}%
(y^{k})\right\Vert +\left\Vert P_{\tilde{\Omega}_{k+1},\tilde{w}_{k+1}}%
(y^{k}-P_{\tilde{\Omega}_{k+1},\tilde{w}_{k+1}}(z^{k}))\right\Vert \nonumber\\
&  \leq(\varepsilon/4)(k_{1}+1)^{-1}+k(\varepsilon/4)(k_{1}+1)^{-1}\nonumber\\
&  =(\varepsilon/4)(k_{1}+1)^{-1}(k+1).
\end{align}
Thus, we have shown by induction that (\ref{eq:6.24}) holds for all
$k=0,1,\dots,k_{1}$ and, in particular, that $\left\Vert y^{k_{1}}-z^{k_{1}%
}\right\Vert \leq\varepsilon/4,$ which implies that%
\begin{equation}
d(x^{k},C)=d(y^{k_{1}},C)\leq\left\Vert y^{k_{1}}-z^{k_{1}}\right\Vert
+d(z^{k_{1}},C)\leq\varepsilon/2,
\end{equation}
and Theorem \ref{thm:thm5.1} is proved. $\blacksquare$

\section{Main convergence result for the string-\newline averaging projected
subgradient method\label{sec:convergence}}

We will make use of the following bounded regularity condition, see
\cite[Definition 5.1]{bb96}, formulated in Hilbert space $H$.

\begin{condition}
\label{cond:A}For each $\varepsilon>0$ and each $M>0$ there exists
$\delta=\delta(\varepsilon,M)>0$ such that for each $x\in B(0,M)$ satisfying
$d(x,C_{i})\leq\delta$, $i=1,2,\dots,m,$ the inequality $d(x,C)\leq
\varepsilon$ holds.
\end{condition}

It follows from \cite[Proposition 5.4(iii)]{bb96} (see also \cite[Proposition
5]{cz12}) that if the Hilbert space $H$ is finite-dimensional then Condition
\ref{cond:A} holds.

Our main convergence result for the string-averaging projected subgradient
method Algorithm \ref{alg:sa-psm} is the following theorem.

\begin{theorem}
{\label{thm:7.1}} Let the assumptions of Theorem \ref{thm:thm5.1} hold and
assume that $\phi:R^{J}\rightarrow R$ is a convex function. Let%
\begin{equation}
\{\alpha_{k}\}_{k=0}^{\infty}\subset(0,1],\text{ such that}\;\lim
_{k\rightarrow\infty}\alpha_{k}=0\text{ and}\;\sum_{k=0}^{\infty}\alpha
_{k}=\infty, \label{eq:7.3}%
\end{equation}
and let $\varepsilon\in(0,1)$. Then there exist an integer $K$ and a real
number $\bar{L}$ such that for any sequence $\{x^{k}\}_{k=0}^{\infty}$,
generated by Algorithm \ref{alg:sa-psm} with $\{(\Omega_{k},w_{k}%
)\}_{k=1}^{\infty}\subset\mathcal{M}_{\ast\ast}$, the inequalities%
\begin{equation}
d(x^{k},SOL(\phi,C))\leq\varepsilon{\text{ and }}\phi(x^{k})\leq
\operatorname{inf}\{\phi(z)\mid z\in SOL(\phi,C)\}+\varepsilon\bar{L}%
\end{equation}
hold for all integers $k\geq K$.
\end{theorem}

It is well-known that the function $\phi$ is continuous due to its convexity.
Clearly, $\phi$ is Lipschitz on bounded subsets of $R^{J}$, therefore, since
$C$ is bounded, by the assumption (\ref{eq:boundedness}), there exists a point
$x\in SOL(\phi,C),$ i.e., $SOL(\phi,C)$ is nonempty. Furthermore, there exists
a number $\bar{L}>1$ such that%
\begin{equation}
|\phi(z^{1})-\phi(z^{2})|\leq\bar{L}||z^{1}-z^{2}||{\text{ for all }}%
z^{1},z^{2}\in B(0,3M+2).\label{eq:Lbar}%
\end{equation}

We will need the following lemma.

\begin{lemma}
{\label{lem-8.3}}Let $\bar{x}\in SOL(\phi,C)$ and let $\Delta\in(0,1]$
$\alpha>0$ and $x\in R^{J}$ satisfy%
\begin{equation}
\left\Vert x\right\Vert \leq3M+2,\;\phi(x)>\phi(\bar{x})+\Delta.
\label{eq:8.3}%
\end{equation}
Further, let $v\in\partial\phi(x)$ and $(\Omega,w)\in\mathcal{M}_{\ast\ast}.$
Then $v\not =0$ and%
\begin{equation}
y:=P_{\Omega,w}\left(  x-\alpha||v||^{-1}v\right)
\end{equation}
satisfies%
\begin{equation}
\Vert y-\bar{x}\Vert^{2}\leq\Vert x-\bar{x}\Vert^{2}-2\alpha(4\bar{L}%
)^{-1}\Delta+\alpha^{2},
\end{equation}
where $\bar{L}$ is as in (\ref{eq:Lbar}). Moreover,
\begin{equation}
d(y,SOL(\phi,C))^{2}\leq d(x,SOL(\phi,C))^{2}-2\alpha(4\bar{L})^{-1}%
\Delta+\alpha^{2}.
\end{equation}

\end{lemma}

\begin{proof}
From (\ref{eq:8.3}) $v\not =0$. For $\bar{x}\in SOL(\phi,C)$, we have, by
(\ref{eq:Lbar}) and (\ref{eq:boundedness}), that for each $z\in B(\bar
{x},4^{-1}\Delta\bar{L}^{-1})$,%
\begin{equation}
\phi(z)\leq\phi(\bar{x})+\bar{L}||z-\bar{x}||\leq\phi(\bar{x})+4^{-1}\Delta.
\end{equation}
Therefore, (\ref{eq:8.3}) and $v\in\partial\phi(x)$, imply that%
\begin{equation}
\left\langle v,z-x\right\rangle \leq\phi(z)-\phi(x) \leq-(3/4)\Delta\text{ for
all }z\in B(\bar{x},4^{-1}\Delta\bar{L}^{-1}).
\end{equation}
From this inequality we deduce that%
\begin{equation}
\left\langle ||v||^{-1}v,z-x\right\rangle <0{\text{ for all }}z\in B(\bar
{x},4^{-1}\Delta\bar{L}^{-1}),
\end{equation}
or, setting $\bar{z}:=\bar{x}+4^{-1}\bar{L}^{-1}\Delta\Vert v\Vert^{-1}v,$
that%
\begin{equation}
0>\left\langle ||v||^{-1}v,\bar{z}-x\right\rangle =\left\langle ||v||^{-1}%
v,\bar{x}+4^{-1}\bar{L}^{-1}\Delta\Vert v\Vert^{-1}v-x\right\rangle .
\end{equation}
This leads to
\begin{equation}
\left\langle ||v||^{-1}v,\bar{x}-x\right\rangle <-4^{-1}\bar{L}^{-1}\Delta.
\end{equation}
Putting $y_{0}:=x-\alpha\Vert v\Vert^{-1}v,$ we arrive at%
\begin{align}
\Vert y_{0}-\bar{x}\Vert^{2}  &  =\Vert x-\alpha\Vert v\Vert^{-1}v-\bar
{x}\Vert^{2}\nonumber\\
&  =\Vert x-\bar{x}\Vert^{2}-2\left\langle x-\bar{x},\alpha\Vert v\Vert
^{-1}v\right\rangle +\alpha^{2}\nonumber\\
&  \leq\Vert x-\bar{x}\Vert^{2}-2\alpha(4\bar{L})^{-1}\Delta+\alpha^{2}.
\end{align}
From all the above we obtain%
\begin{align}
\Vert y-\bar{x}\Vert^{2}  &  =\Vert P_{\Omega,w}(y_{0})-\bar{x}\Vert^{2}%
\leq\Vert y_{0}-\bar{x}\Vert^{2}\nonumber\\
&  \leq\Vert x-\bar{x}\Vert^{2}-2\alpha(4\bar{L})^{-1}\Delta+\alpha^{2},
\end{align}
which completes the proof of the lemma.
\end{proof}

Now we present the proof of Theorem \ref{thm:7.1}.

\textbf{Proof of Theorem \ref{thm:7.1}. }Fix an $\bar{x}\in SOL(\phi,C).$ It
is not difficult to see that there exists a number $\varepsilon_{0}%
\in(0,\varepsilon/4)$ such that for each $x\in R^{J}$ satisfying
$d(x,C)\leq\varepsilon_{0}$ and $\phi(x)\leq\phi(\bar{x})+\varepsilon_{0}$ we
have%
\begin{equation}
d(x,SOL(\phi,C))\leq\varepsilon/4. \label{eq:P2}%
\end{equation}

%Choose a positive $\varepsilon_{1}$ for which $\varepsilon_{1}<(8\bar{L}%
%)^{-1}\varepsilon_{0}.$ By (\ref{eq:7.3}) there is an integer $n_{2}>n_{1}$
%such that%
%\begin{equation}
%\alpha_{k}\leq\varepsilon_{1}(32)^{-1},{\text{ for all }}k>n_{2}%
%,\label{eq:8.16}%
%\end{equation}
%and so, there is an integer $n_{0}>n_{2}+4$ such that%
%\begin{equation}
%\sum_{k=n_{2}}^{n_{0}-1}\alpha_{k}>8(2M+1)^{2}\bar{L}\varepsilon_{0}%
%^{-1}.\label{eq:8.17}%
%\end{equation}

Since $\{x^{k}\}_{k=0}^{\infty}$ is generated by Algorithm \ref{alg:sa-psm}
and $\{(\Omega_{k},w_{k})\}_{k=1}^{\infty}\subset\mathcal{M}_{\ast\ast}$ we
know, from (\ref{eq:alg-sa-psm-2}) and (\ref{eq:1.8}), that%
\begin{equation}
\left\Vert x^{k}-P_{\Omega_{k},w_{k}}(x^{k-1})\right\Vert \leq\alpha
_{k-1},\text{ for all }k\geq1, \label{eq:8.23}%
\end{equation}
holds. Thus, by Theorem \ref{thm:thm5.1} and (\ref{eq:7.3}), there exists an
integer $n_{1}$ such that%
\begin{equation}
d(x^{k},C)\leq\varepsilon_{0},\text{ for all }k\geq n_{1}. \label{eq:8.24}%
\end{equation}
This, along with (\ref{eq:boundedness}), guarantees that%
\begin{equation}
\Vert x^{k}\Vert\leq M+1,{\text{ for all }}k\geq n_{1}. \label{eq:8.25}%
\end{equation}

Choose a positive $\varepsilon_{1}$ for which $\varepsilon_{1}<(8\bar{L}%
)^{-1}\varepsilon_{0}.$ By (\ref{eq:7.3}) there is an integer $n_{2}>n_{1}$
such that%
\begin{equation}
\alpha_{k}\leq\varepsilon_{1}(32)^{-1},{\text{ for all }}k>n_{2},
\label{eq:8.16}%
\end{equation}
and so, there is an integer $n_{0}>n_{2}+4$ such that%
\begin{equation}
\sum_{k=n_{2}}^{n_{0}-1}\alpha_{k}>8(2M+1)^{2}\bar{L}\varepsilon_{0}^{-1}.
\label{eq:8.17}%
\end{equation}

We show now that there exists an integer $p\in\lbrack n_{2}+1,n_{0}]$ such
that $\phi(x^{p})\leq\phi(\bar{x})+\varepsilon_{0}$. Assuming the contrary
means that for all $k\in\lbrack n_{2}+1,n_{0}]$,
\begin{equation}
\phi(x^{k})>\phi(\bar{x})+\varepsilon_{0}. \label{eq:8.26}%
\end{equation}
By (\ref{eq:8.26}), (\ref{eq:7.3}), (\ref{eq:8.25}) and using Lemma
\ref{lem-8.3}, with $\Delta=\varepsilon_{0}$, $\alpha=\alpha_{k}$, $x=x^{k}$,
$y=x^{k+1}$, $v=s^{k}$, for all $k\in\lbrack n_{2}+1,n_{0}]$, we get%
\begin{equation}
d(x^{k+1},SOL(\phi,C))^{2}\leq d(x^{k},SOL(\phi,C))^{2}-2\alpha_{k}(4\bar
{L})^{-1}\varepsilon_{0}+\alpha_{k}^{2}.
\end{equation}
According to the choice of $\varepsilon_{1}$ and by (\ref{eq:8.16}) this
implies that for all $k\in\lbrack n_{2}+1,n_{0}]$,
\begin{align}
d(x^{k},SOL(\phi,C))^{2}-d(x^{k+1},SOL(\phi,C))^{2}  &  \geq\alpha_{k}%
[(2\bar{L})^{-1}\varepsilon_{0}-\alpha_{k}]\nonumber\\
&  \geq\alpha_{k}(4\bar{L})^{-1}\varepsilon_{0},
\end{align}
which, together with (\ref{eq:8.25}) and (\ref{eq:boundedness}), gives%
\begin{align}
(2M+1)^{2}  &  \geq d(x^{n_{2}+1},SOL(\phi,C))^{2}\nonumber\\
&  \geq\sum_{k=n_{2}+1}^{n_{0}}\left(  d(x^{k},SOL(\phi,C))^{2}-d(x^{k+1}%
,SOL(\phi,C))^{2}\right) \nonumber\\
&  \geq(4\bar{L})^{-1}\varepsilon_{0}\sum_{k=n_{2}+1}^{n_{0}}\alpha_{k}%
\end{align}
and%
\begin{equation}
\sum_{k=n_{2}+1}^{n_{0}}\alpha_{k}\leq(2M+1)^{2}4\bar{L}\varepsilon_{0}^{-1}.
\end{equation}
This contradicts (\ref{eq:8.17}), proving that there is an integer
$p\in\lbrack n_{2}+1,n_{0}]$ such that $\phi(x^{p})\leq\phi(\bar
{x})+\varepsilon_{0}$. Thus, by (\ref{eq:8.24}) and (\ref{eq:P2}),%
\begin{equation}
d(x^{p},SOL(\phi,C))\leq\varepsilon/4.
\end{equation}
We show that for all $k\geq p$, $d(x^{k},SOL(\phi,C))\leq\varepsilon$.
Assuming the contrary that%
\begin{equation}
\text{there exists a }q>p\text{ such that }d(x^{q},SOL(\phi,C))>\varepsilon.
\label{eq:8.31}%
\end{equation}
We may assume, without loss of generality, that%
\begin{equation}
d(x^{k},SOL(\phi,C))\leq\varepsilon,{\text{ for all }}p\leq k<q.
\label{eq:8.32}%
\end{equation}
One of the following two cases must hold: (i) $\phi(x^{q-1})\leq\phi(\bar
{x})+\varepsilon_{0},$ or (ii) $\phi(x^{q-1})>\phi(\bar{x})+\varepsilon_{0}.$
In case (i), since $p\in\lbrack n_{2}+1,n_{0}],$ (\ref{eq:8.24}),
(\ref{eq:8.25}) and (\ref{eq:P2}) show that%
\begin{equation}
d(x^{q-1},SOL(\phi,C))\leq\varepsilon/4.
\end{equation}
Thus, there is a point $z\in SOL(\phi,C)$ such that $\left\Vert x^{q-1}%
-z\right\Vert <\varepsilon/3.$ Using this fact and (\ref{eq:8.23}),
(\ref{eq:1.8}), (\ref{eq:1.9}) and (\ref{eq:8.16}), yields%
\begin{align}
\left\Vert x^{q}-z\right\Vert  &  \leq\left\Vert x^{q}-P_{\Omega_{q},w_{q}%
}(x^{q-1})\right\Vert +\left\Vert P_{\Omega_{q},w_{q}}(x^{q-1})-z\right\Vert
\nonumber\\
&  \leq\alpha_{q-1}+\left\Vert x^{q-1}-z\right\Vert \leq\varepsilon
/4+\varepsilon/3,
\end{align}
proving that $d(x^{q},SOL(\phi,C))\leq\varepsilon.$ This contradicts
(\ref{eq:8.31}) and implies that case (ii) must hold, namely that
$\phi(x^{q-1})>\phi(\bar{x})+\varepsilon_{0}$. This, along with (\ref{eq:8.25}%
), (\ref{eq:8.16}), the choice of $\varepsilon_{1}$, (\ref{eq:8.32}) and Lemma
\ref{lem-8.3}, with $\Delta=\varepsilon_{0}$, $\alpha=\alpha_{q-1}$,
$x=x^{q-1}$, $y=x^{q}$, shows that%
\begin{align}
d(x^{q},SOL(\phi,C))^{2}  &  \leq d(x^{q-1},SOL(\phi,C))^{2}-2\alpha
_{q-1}(4\bar{L})^{-1}\varepsilon_{0}+\alpha_{q-1}^{2}\nonumber\\
&  \leq d(x^{q-1},SOL(\phi,C))^{2}-\alpha_{q-1}(2\bar{L})^{-1}\varepsilon
_{0}-\alpha_{q-1})\nonumber\\
&  \leq d(x^{q-1},SOL(\phi,C))^{2}\leq\varepsilon^{2},
\end{align}
namely, that $d(x_{q},C_{min})\leq\epsilon.$ This contradicts (\ref{eq:8.31}),
proving that, for all $k\geq p$, $d(x^{k},SOL(\phi,C))\leq\varepsilon$.
Together with (\ref{eq:boundedness}) and (\ref{eq:Lbar}) this implies that,
for all $k\geq n_{0}$,%
\begin{equation}
\phi(x^{k})\leq\operatorname{inf}\{\phi(z)\mid z\in SOL(\phi,C)\}+\varepsilon
\bar{L}\}
\end{equation}
and the proof of Theorem \ref{thm:7.1} is complete.\bigskip

\textbf{Acknowledgments}. We greatly appreciate the constructive comments of
two anonymous reviewers which helped to improve the paper. The work of the
first author was partially supported by the United States-Israel Binational
Science Foundation (BSF) Grant number 200912 and US Department of Army Award
number W81XWH-10-1-0170.

\end{document}